\documentclass[a4paper,11pt]{article}    

\usepackage{fixltx2e}
\usepackage{float}
\usepackage{amssymb,amsmath}
\usepackage{amsthm}
\usepackage{mathtools}
\usepackage{fixmath}
\usepackage{caption}
\usepackage{subcaption}
\usepackage[final]{graphicx}
\usepackage{xcolor}
\usepackage{braket}
\usepackage{booktabs}
\usepackage{tikz,pgfplots}
\usepackage[english]{babel}
\usepackage[T1]{fontenc}
\usepackage{lmodern}
\usepackage[final,stretch=10]{microtype}
\usepackage[colorlinks=true,linkcolor=black,urlcolor=black,citecolor=black,pdftex]{hyperref}
\hypersetup{final} 
\usepackage{todonotes}

\newcommand{\R}{\mathbb{R}}

\renewcommand{\phi}{\varphi}
\newcommand{\eps}{\varepsilon}
\renewcommand{\d}{\mathop{}\!\mathrm{d}}

\DeclareMathOperator{\supp}{supp}

\renewcommand\div{\operatorname{div}}

\theoremstyle{plain}
\newtheorem{theorem}{Theorem}
\newtheorem{proposition}[theorem]{Proposition}
\newtheorem{lemma}[theorem]{Lemma}

\theoremstyle{definition}
\newtheorem{definition}{Definition}

\theoremstyle{remark}

\pgfplotsset{every tick label/.append style={font=\tiny}}
\begin{document}

\title{An inverse problem involving a viscous Eikonal equation with applications in electrophysiology\thanks{The authors were supported by the ERC advanced grant 668998 (OCLOC) under the EU’s H2020 research program.}}


\author{Karl Kunisch \thanks{Institute for Mathematics and Scientific Computing, University of Graz, Heinrichstrasse 36, A-8010 Graz, Austria, and Radon Institute, Austrian Academy of Science, (karl.kunisch@uni-graz.at).}
    \and Philip Trautmann \thanks{Institute for Mathematics and Scientific Computing, University of Graz, Heinrichstrasse 36, A-8010 Graz, Austria,   (philip.trautmann@uni-graz.at).}}


\maketitle

\begin{abstract}
In this work we discuss the reconstruction of cardiac activation instants based on a viscous Eikonal equation from boundary observations. The problem is formulated as an least squares problem and solved by a projected version of the Levenberg Marquardt method. Moreover, we analyze the wellposeness of the state equation and derive the gradient of the least squares functional with respect to the activation instants. In the numerical examples we also conduct an experiment in which the location of the activation sites and the activation instants are reconstructed jointly based on an adapted version of the shape gradient method from \cite{kunisch2019inverse}. We are able to reconstruct the activation instants as well as the locations of the activations with high accuracy relative to the noise level.
\end{abstract}

\section{Introduction}

This work is concerned with an inverse problem in cardiac electrophysiology. In particular, the activation instants of the excitation wave in the myocardium are estimated from the arrival times of the wave at the epicardium. To briefly explain the problem we recall that  the electro-physiologic activity of the heart is often modeled using the bidomain equations, whose numerical solution is very expensive. If one is only interested in the activation times $T$ of the tissue, the bidomain model can be reduced to the simpler viscous Eikonal equation given, for instance, in the form
\begin{equation}\label{state_eq_eps_intro}
\left\{
\begin{aligned}
-\eps\operatorname{div}(M\nabla T)+\sqrt{M\nabla T\cdot \nabla T}&=1&\text{in}~\Omega,\\
T&=u_i&\text{on}~\Gamma_i,\quad i=1,\ldots,n\\
\eps \nabla T\cdot n&=0&\text{on}~\Gamma_N.
\end{aligned}\right.
\end{equation}
The domain $\Omega$ models the computational geometry of the heart. The epicardium of the heart is denoted by $\Gamma_N$ and the boundaries of the activation regions (activation sites) by $\Gamma_i$. The matrix $M$ describes the fiber orientation of the heart tissue and the values $u_i\in \mathbb R $ are the activation instants in the activation regions. On the basis of this model we formulate the inverse problem in the following form
\begin{equation}\label{inverse_prob_eps_intro}
\min_{u}J(u):=\frac 1 2\int_{\Gamma_N}(T(u)-z)^2~\mathrm d x\quad \text{subject to}~\eqref{state_eq_eps_intro},
\end{equation}
where $z$ is the measured data on the epicardium. Problem \eqref{inverse_prob_eps_intro} constitutes an inverse problem for the activation instants $u_i$. 
While in the analysis part we focus  on reconstructing the activation instants from measurements of the activation time $T$ on the surface of the computational domain $\Omega$, in the numerical section we demonstrate that the activation instants and
the location of the activation sites can be reconstructed simultaneously. 

To briefly  comment on the physiological background of this research, we point out that
computational models of cardiac function are increasingly considered as a clinical research tool.
For the understanding of the driving mechanism of cardiac electro-mechano-fluidic function,
the sequence of electrical activations is of key importance.
Computer models intended for clinical applications
must be parameterized in a patient-specific manner
to approximate the electrical activation sequence in a given patient's heart, which necessitates to solving inverse problems to identify patient specific parameters.
Anatomical \cite{demoulin72:_histo,ono09:_morphological} as well as early experimental mapping studies \cite{durrer70:_excitation_human},
using \emph{ex vivo} human hearts
provided evidence that electrical activation in the left ventricle (LV),
i.e.\ the main pumping chamber that drives blood into the circulatory system,
is initiated by the His-Purkinje system \cite{haissaguerre16:_His_purkinje}
at several specific sites of earliest activation (root points)
which are located at the endocardial (inner) surface of the LV.
In a first approximation it can be assumed that the healthy human LV is activated
at these root points
by a tri-fascicular conduction system \cite{rosenbaum69:_trifascicular}
consisting of three major fascicles referred to as anterior, septal and posterior fascicle.
Size and location of these patches
as well as the corresponding instants of their activation are key determinants shaping the activation sequence of the left ventricle.
Since the His-Purkinje system is highly variable in humans,
there is significant interest in inverse methods for identifying these sites and activation instants, ideally non-invasively.

 To briefly outline the paper, first we give a sufficient condition for the  well-posedness of the elliptic PDE using the Schauder fixed point theorem and the maximum principle. The activation instants enter the state equation as constant Dirichlet boundary conditions on the surface of the activation regions. Then we calculate the gradient of the least squares cost functional with respect to these activation instants. It can be expressed in terms of the normal derivative of the solution to the adjoint state equation on the surface of activation sites. Therefore we also analyze the well-posedness of the adjoint and linearized state equations. Finally, we propose to solve the least squares problem using the projected Levenberg Marquardt method.\\
In our numerical experiments we first consider only the reconstruction of the activation instants using the  proposed Levenberg Marquardt method. In the second numerical example we perform the joint reconstruction of the activation sites and the activation instants. The activation sites are reconstructed by means of an adapted version of the shape gradient method introduced in \cite{kunisch2019inverse} together with a projected gradient method for the reconstruction of the activation instants. The numerical examples illustrate the feasibility of the approach and are  carried out on the 2D unit square with artificial data.

\section{Problem statement}
Let $U\subset\R^d$, with $d=2$ or $d=3$ be an open domain and $\Gamma_N=\partial U$ its boundary. In the physiological context it represents the cardiac domain. Within $U$ we consider a family of open subdomains $\{\omega_i\}_{i=1}^n$ and we set  $\Gamma_i=\partial\omega_i$. These boundaries constitutes the surface from where the activation spreads.
 Then we define $\Omega=U\setminus \bigcup_{i=1}^n\omega_i$ which is our  mathematical and computational cardiac domain, with boundary $\partial\Omega=\Gamma_N\cup \bigcup_{i=1}^n\Gamma_i$. Note that $\Omega$ is connected but not simply connected.
  Let us choose a parameter  $\eps>0$, and fix  $z\in H^{1/2}(\Gamma_N)$, which represents the epicardial input data.

   With these specifications we consider the following problem:
\begin{equation}\label{inverse_prob_eps}
\min_{u\in U_{ad}}J(u)=\frac 1 2\int_{\Gamma_N}(T(x)-z(x))^2\d x
\end{equation}
subject to the viscous Eikonal equation
\begin{equation}\label{state_eq_eps}
\left\{
\begin{aligned}
-\eps\div(M\nabla T)+\sqrt{\beta+|\nabla T|_M^2}&=1&\text{in}~\Omega\\
T&=u_i&\text{on}~\Gamma_i,\quad i=1,\ldots,n\\
\eps M\nabla T\cdot n&=0&\text{on}~\Gamma_N
\end{aligned}\right.
\end{equation}
where $\beta\in[0,1]$,  $n$ is the unit normal on $\Gamma_N$, and $|\nabla T|_M^2={ M\nabla T\cdot \nabla T}$.
Further $u={\text{col}}(u_1,\dots, u_n)\in U_{ad}$ which is a closed and convex set in $\R^n$. The function $T$ stands for the activation time, and the matrix $M$ models the cardiac conduction velocity.

For the mathematical description of the excitation process in the myocardium Eikonal equations are a well-established procedure. Notably we refer to \cite[Section 5]{colli90} where, on the basis of the bidomain equations,  a singular perturbation technique with respect to  the thickness of the myocardial wall and the time taken by the excitation wave front to cross the heart wall is carried out to arrive at various models for the Eikonal equation which differ by the nonlinear term. The two versions which are advocated in that paper  and for which numerical comparisons are carried out are  $|\nabla T|_M^2$ and $\sqrt{|\nabla T|_M^2}$. It is stated there that the model involving $\sqrt{|\nabla T|_M^2}$ is better for wavefront propagation and collision. In earlier work \cite{kunisch2019inverse} we have used $|\nabla T|_M^2$ and solved the inverse shape  problem of identifying the centers of spherical subdomains $\omega_i$ from epicardial data $z$.

\section{Well posedness of the viscous Eikonal equation}
We assume that the boundaries of $\Omega$ are chosen such that the equation
\begin{equation}\label{poisson}
\left\{
\begin{aligned}
-\eps\div(M\nabla T_H)&=\tilde f&\text{in}~\Omega\\
T_H&=0&\text{on}~\Gamma_i,\quad i=1,\ldots,n\\
\eps M\nabla T_H\cdot n &=0&\text{on}~\Gamma_N.
\end{aligned}\right.
\end{equation}
has a unique solution $T_H\in H^2(\Omega)$ for any $\tilde f\in L^2(\Omega)$. Moreover we assume that $M\in W^{1,\infty}(\Omega)^{d\times d}$ and that $M(x)v\cdot v\geq \alpha |v|^2$ for a.e. $x\in \Omega$ holds. Further, for any $u\in\R^n$ we assume the existence of $g\in W^{2,6}(\Omega)$ with $g|_{\Gamma_i}=u_i$ for $i=1,\ldots,n$, $g$  vanishing in a neighbourhood of $\Gamma_N$,  and $\|g\|_{W^{2,6}(\Omega)}\leq c|u|_{\R^n}$, with $c$ independent of $u$. For example $g=\sum_{i=1}^nu_ig_i$ can be chosen where the functions $g_i$ are chosen as smooth bump functions which are  equal to $1$ on $\bar \omega_i$, vanish near $\Gamma_N$ and have the property $\supp(g_i)\cap \supp(g_j)=\emptyset$ for all $i,j=1,\ldots,n$. Moreover for $\tilde T:=T_H+g\in H^2(\Omega)$,  we have $\tilde T|_{\Gamma_i}=u_i$ for all $i=1,\ldots,n$ and
\[
\eps \int_\Omega M\nabla T_H\cdot \nabla v~\mathrm dx=\int_\Omega \tilde fv+ \eps \div(M\nabla g)v~\mathrm dx
\]
for all $v\in V:=  H^1_0(\Omega\cup\Gamma_N)=\{v\in H^1(\Omega)|\quad v|_{\Gamma_i}=0,~i=1,\ldots,n\}$. In the subsequent developments \eqref{poisson} will be used with $\tilde f$ replaced by \[-\sqrt{\beta+|\nabla (T_H +g)|_M^2}+ \varepsilon\div(M\nabla g)+1\].
\begin{theorem}\label{theo:existence}
For $\eps>0$ sufficiently large \eqref{state_eq_eps} has a unique solution
\[T\in W^{2,6}(\Omega).\]
Moreover there exists a constant $c$, independent of $u\in \R^{n}$, and $\beta\in[0,1]$ such that $\|T\|_{W^{2,6}} \le c(1+ |u|)$.
\end{theorem}
\begin{proof}
1.Existence: Let $T_H\in H^1_0(\Omega\cup\Gamma_N)$ be fixed. Then we set
\[f(T_H)(x):=-\sqrt{\beta+|\nabla(T_H(x)+g(x))|_M^2}\]
with $\beta\in [0,1]$. Since $T_H\in V$, $g\in W^{2,6}(\Omega)$ and $M\in W^{1,\infty}(\Omega)^{d\times d}$ it follows that $f(T_H)\in L^2(\Omega)$. Now let $w\in H^2(\Omega)$ be the unique solution of
\begin{equation}\label{poisson_aux}
\left\{
\begin{aligned}
-\eps\div(M\nabla w)&=f(T_H)+\div(M\nabla g)+1&\text{in}~\Omega\\
w&=0&\text{on}~\Gamma\\
\eps M\nabla w\cdot n &=0&\text{on}~\Gamma_N
\end{aligned}\right.
\end{equation}
with the estimate
\[
\|w\|_{H^2(\Omega)}\leq c(\|f(T_H)\|_{L^2(\Omega)}+|u|+1).
\]
Thus we can define the operator $G\colon V \rightarrow H^2(\Omega) \subset V $, $G\colon T_H\mapsto w$ which satisfies the inequality
\begin{equation}\label{bound_G}
\|G(T_H)\|_{H^2(\Omega)}\leq c(M,\eps)(\|T_H\|_{V}+|u|+1),
\end{equation}
with $c(M,\eps)$ independent of $\beta\in [0,1]$ and $T_H$. In the following we shall utilize Schaefer's fixed point theorem in order to prove that $G$ has a fixed point. At first we prove that $G\colon V\rightarrow V$ is continuous and compact. Let $\{T_{H,k}\}_k\subset V$ be a convergent  sequence with limit $T_H$ in $V$. We set $w_k:=G(T_{H,k})$ and have
\[
\sup_k\|w_k\|_{H^2(\Omega)}<\infty
\]
according to \eqref{bound_G}.
The compact embedding of $H^2(\Omega)\cap V$ in $V$ implies the existence of a subsequence $\{w_k\}$ and of a $w\in V$ with $w_k\rightarrow w$ in $V$. By taking the limit in the weak formulation of \eqref{poisson_aux} we see that $G(T_H)=w$. Thus $G\colon V\rightarrow V$ is continuous. A similar argument shows that $G\colon V \rightarrow V$ is compact. In order to apply Schaefer fixed point theorem we have to further show that the set
\[
\set{T\in V|T=\lambda G(T)~\text{for some}~0<\lambda\leq1}
\]
is bounded in V. Let $T_H\in V$ be such that $T_H=\lambda G(T_H)$ for some $0<\lambda\leq1$. Then we have
\[
-\eps\div(M\nabla T_H)=\lambda(f(T_H)+\eps \div(M\nabla g)+1)\quad\text{a.e. in }\Omega.
\]
Multiplying this equation with $T_H$ and integrating over $\Omega$, we obtain by Young's inequality and the fact that   $0<\lambda \leq 1$:
\begin{multline*}
\eps\alpha\|\nabla T_H\|_{L^2(\Omega)}^2\leq\eps\int_\Omega M\nabla T_H\cdot\nabla T_H~\mathrm dx\\
 =\lambda\int_{\Omega}(-\sqrt{\beta+|\nabla(T_H+g)|_M^2}+1+\eps \div(M\nabla g))T_H~\mathrm dx\\
 \leq \|M\|_\infty^2\|\nabla T_H\|_{L^2(\Omega)}^2+\frac 32\|T_H\|_{L^2(\Omega)}^2+ \frac{\eps^2}{2}\|\div(M\nabla g)\|_{L^2(\Omega)}^2\\
 +\|M\|_\infty^2\|\nabla g\|_{L^2(\Omega)}^2+\frac{|\Omega|}{2}(\beta+1)\\
 \leq c(M)\,(\|\nabla T_H\|_{L^2(\Omega)}^2+\varepsilon^2|u|^2+1+\beta),
\end{multline*}
with $c(M)$ independent of $\lambda$ and $\varepsilon$. Thus if $\eps$ is sufficiently large, we have $\|T_H\|_{V}\leq \tilde c(M,\eps)(1+|u|) $, for some constant $\tilde c(M,\eps)$ independent of $\lambda\in(0,1]$ and $\beta\in[0,1]$. Then  Schaefer's fixed point theorem can be applied to $G$ and yields the existence of an element $T_H\in V$ with $G(T_H)=T_H$ which is a solution of \eqref{poisson_aux}. Setting $T=T_H+g$ we have obtained a solution to  \eqref{state_eq_eps}, for which by   \eqref{bound_G} we have $|T|_{H^2(\Omega)} \le C(M,\eps)$, with $C(M,\eps)$ independent of $\beta\in [0,1]$.

Moreover, since $\nabla T\in H^1(\Omega)^d$ and thus $\nabla T\in L^6(\Omega)^d$, and since also $g\in W^{2,6}(\Omega)$ we have that
\[
-\sqrt{\beta+|\nabla T|_M^2}+1+\varepsilon\div(M\nabla g)\in L^6(\Omega),
\]
and thus
\[\|  T\|_{W^{2,6}(\Omega)}\le \tilde C(M,\eps)(1+|u|) \text { with } \tilde C(M,\eps) {\text {  independent of  }} \beta\in [0,1].
\]\\

2.Uniqueness: Let $T_i\in W^{2,6}(\Omega)$, $i=1,2$ be two solutions of \eqref{state_eq_eps} and define $\delta T=T_1-T_2$. Then $\delta T$ satisfies the equation
\begin{equation}\label{state_diff}
\left\{
\begin{aligned}
-\eps\div(M\nabla \delta T)+\sqrt{\beta+|\nabla T_1|_M^2} - \sqrt{\beta+|\nabla T_2|_M^2}&=0&\text{in}~\Omega\\
\delta T&=0&\text{on}~\Gamma\\
\eps\nabla \delta T\cdot n&=0&\text{on}~\Gamma_N.
\end{aligned}\right.
\end{equation}
Let us define for $(x,v) \in \Omega \times \R^d$ the function
\[
B(x,v):=
\begin{cases}
\frac{M(x)v}{\sqrt{\beta+|v|_{M(x)}^2}}&v\neq 0\\
0&v=0
\end{cases}
\]
It is easy to see, that
\[
B(x,\bar v)\cdot(v-\bar v)\leq \sqrt{\beta+|v|_M^2}-\sqrt{\beta+|\bar v|_M^2}
\]
holds. Indeed, in case $\beta +|\bar v|_M^2 =0$ the inequality is correct by the definition of $B$.  Otherwise we have
\begin{multline*}
\frac{M\bar v\cdot(v-\bar v)}{\sqrt{\beta+|\bar v|_{M}^2}}
=\frac{M\bar v\cdot v+\beta}{\sqrt{\beta+|\bar v|_M^2}}-\frac{\beta+|\bar v|_M^2}{\sqrt{\beta+|\bar v|_M^2}}\\
\leq \frac{\sqrt{\beta+|v|_M^2}\sqrt{\beta+|\bar v|_M^2}}{\sqrt{\beta+|\bar v|_M^2}}-\sqrt{\beta+|\bar v|_M^2}
=\sqrt{\beta+|v|_M^2}-\sqrt{\beta+|\bar v|_M^2}.
\end{multline*}
Here we have used that $\left(
                         \begin{array}{cc}
                           M & 0 \\
                           0 & 1 \\
                         \end{array}
                       \right)$
defines a scalar product for the vectors $(v,\sqrt{\beta})$.
Alternatively we can note that $B(x,v)$ is an element of the subdifferential of the convex function $v\to \sqrt{\beta + |v|_M^2}$. Thus we have
\[
B(x,\nabla T_2)\cdot\nabla\delta T\leq \sqrt{\beta+|\nabla T_1|_M^2}-\sqrt{\beta+|\nabla T_2|_M^2}.
\]
Consequently
\[
-\eps\div(M\nabla \delta T)+B(x,\nabla T_2)\cdot \nabla\delta T\leq 0,
\]
where $B(x,\nabla T_2)\in L^\infty(\Omega)^d$, since $T_2$ is an element of $W^{2,6}(\Omega)$.
Then the maximum principle implies that $\delta T \le 0$ in $\Omega$, see \cite[Theorem 3.27]{Troianiello87}. Exchanging the roles of $T_1$ and  $T_2$ in the above argument leads to $\delta T \ge 0$ in $\Omega$, and consequently to $\delta T=0$, which implies the desired uniqueness.
\end{proof}
This proof is inspired from \cite[Section 9.2, Theorem 5]{evans1998partial}.
Henceforth it will be assumed that $\eps$ is large enough so that the solution to \eqref{state_eq_eps} according to Theorem \ref{theo:existence} exists.
\begin{theorem}
We have
\[
T_\beta\to T_0\quad \text{in}~ H^2(\Omega),
\]
where $T_\beta$ denotes the solution to \eqref{state_eq_eps} as a function of $\beta$.
\end{theorem}
\begin{proof}
By Theorem \ref{theo:existence} the family $\{T_H^\beta\}_{\beta\in [0,1]}$ is bounded in $H^2(\Omega)\cap V$ and hence there exists a subsequence, denoted in the same manner, and $\hat T_H \in H^2(\Omega) \cap V$ such that $T_H^\beta \rightharpoonup \hat T_H$ in $H^2(\Omega)$ and $T_H^\beta \to \hat T_H$ in $V$. Thus we can pass to the limit in
\begin{multline*}
\int_\Omega\eps M\nabla T_H^\beta\cdot\nabla \varphi+\sqrt{\beta+|\nabla (T_H^\beta+g)|_M^2}\varphi~\mathrm dx\\
=\int_\Omega(1+\varepsilon\div(M\nabla g))\varphi~\mathrm dx, \text{ for all } \varphi \in V
\end{multline*}
to obtain that
\begin{equation*}
\int_\Omega\eps M \nabla \hat T_H\cdot\nabla \varphi +|\nabla (\hat T_H+g)|_M\varphi~\mathrm dx=\int_\Omega(1+\varepsilon\div(M\nabla g)) \varphi~\mathrm dx, \text{ for all } \varphi \in V.
\end{equation*}
Moreover, by the trace theorem $\hat T_H=0\text{ on }~\Gamma_i, \text{ for } i=1,\ldots,n$. Now we set $T_0=\hat T_H+g$.
By uniqueness, asserted in Theorem \ref{theo:existence} we have $\hat T_H = T_H$, where $T_H$ is the homogenous solution for $\beta =0$ from Theorem \ref{theo:existence}, and thus the whole family $T_\beta=T_H^\beta+g$ converges to $T_0$ in $V$.
Moreover we have
\begin{multline*}
\int_\Omega\eps^2 |\div(M\nabla (T_\beta - T_0))|^2~\mathrm dx=
\int_\Omega\eps^2 |\div(M\nabla (T_H^\beta - T_H))|^2~\mathrm dx\\
= \int_\Omega(\sqrt{\beta+|\nabla T_\beta|_M^2}-|\nabla T_0|_M)^2~\mathrm dx \to 0
\end{multline*}
for $\beta \to 0^+$. Since $T_H^\beta|_{\Gamma_i},T_H|_{\Gamma_i}=0$ and $\left(\int_\Omega|\div(M\nabla \cdot)|^2~\mathrm dx\right)^{1/2}$ defines an equivalent norm to the $H^2(\Omega)$-norm on $H^2(\Omega)\cap V$, the claim follows.
\end{proof}
\section{Well posedness of the linearized and adjoint state equation}

Throughout the rest of the theoretical part of this work $T\in W^{2,6}(\Omega)$ with $M\nabla T\cdot n|_{\Gamma_n}=0$, and  $\beta \in (0,1]$ are assumed. Further $u$, $r$ and $h$ are chosen arbitrarily in $\R^n$, $L^2(\Omega)$ and $H^{1/2}(\Gamma_N)$, respectively. We analyse the well-posedness of the following equations
\begin{equation}\label{lin_state_eq_eps}
\left\{
\begin{aligned}
-\eps\div(M\nabla \delta T)+\frac{M\nabla T\cdot\nabla \delta T}{\sqrt{\beta+|\nabla T|_M^2}}&=r&\text{in}~\Omega\\
\delta T&= u_i&\text{on}~\Gamma_i,\quad i=1,\ldots,N\\
\eps M\nabla \delta T\cdot n &=0&\text{on}~\Gamma_N.
\end{aligned}\right.
\end{equation}
and
\begin{equation}\label{ad_state_eq_eps1}
\left\{
\begin{aligned}
-\eps\div(M\nabla \phi)-\div\left(\frac{M\nabla T}{\sqrt{\beta+|\nabla T|_M^2}}\phi\right)&=0&\text{in}~\Omega\\
\phi&=0&\text{on}~\Gamma\\
\eps M\nabla \phi\cdot n&=h&\text{on}~\Gamma_N.
\end{aligned}\right.
\end{equation}
 For this purpose we define the bilinear form $B\colon V \times V \rightarrow \R$ by
\[
B(v,\phi):=\eps(M\nabla v,\nabla \phi)_{L^2(\Omega)}+\left(\frac{M\nabla T\cdot \nabla v}{\sqrt{\beta+|\nabla T|_M^2}},\phi\right)_{L^2(\Omega)}
\]
for any $\phi,v\in V$.
We recall the function $g\in W^{2,6}(\Omega)$ defined in the previous section.
\begin{definition}\label{def:tangent}
The function $\delta T=v+g\in H^1(\Omega)$ is called a weak solution of \eqref{lin_state_eq_eps} if $v\in V$ solves the variational equation
\begin{equation}\label{eq:kk1}
B(v,\phi)=\int_{\Omega}\left(\eps\div(M\nabla g)-\frac{M\nabla T\cdot\nabla g}{\sqrt{\beta+|\nabla T|_M^2}}+r\right)\phi~\mathrm dx\quad \forall \phi\in V.
\end{equation}
Analogously  $\phi\in  V$ is called a weak solution of \eqref{ad_state_eq_eps1} if it solves the variational equation
\[
B(v,\phi)=\int_{\Gamma_N}gv~\mathrm ds\quad \forall v\in  V.
\]
\end{definition}
We introduce the operator $\mathcal A\colon V \to V^\ast$ and its adjoint $\mathcal A^\ast\colon V\to V^\ast$ by
\[
\langle \mathcal Av,\phi\rangle =B(v,\phi)=\langle v,\mathcal A^\ast\phi\rangle.
\]
for all $v,\phi\in V$.

\begin{proposition}\label{prop:isom}
The operators $\mathcal A\colon V\to V^\ast$ and $\mathcal A^\ast\colon V \to V^\ast$ are isomorphisms. In particular there exists a constant $C(M,T,\eps)$ such that
\begin{equation}\label{eq:inv}
\|\mathcal A^{-1}\|_{\mathcal L(V^\ast,V)}=\|\mathcal A^{-\ast}\|_{\mathcal L(V^\ast, V)}\leq C(M,T,\eps).
\end{equation}
\end{proposition}
\begin{proof}
The claims follow from a similar argumentation as in the proof of Proposition 2 in \cite{kunisch2019inverse} using Garding's inequality and the weak maximum principle.
\end{proof}
We introduce the space
\[
W=\{v\in H^2(\Omega)|\quad v|_{\Gamma_i}\in \mathbb R,~i=1,\ldots,N,~ M\nabla v\cdot n|_{\Gamma_N}=0\}.
\]
The space $W$ is a closed subspace of $H^2(\Omega)$, since the trace as well as the normal trace operator are continuous.
\begin{proposition}\label{prop:tangent}
Equation \eqref{lin_state_eq_eps} has a unique weak solution which satisfies $\delta T\in W$ and
\begin{equation}\label{eq:kk2}
\|\delta T\|_{H^2(\Omega)}\leq C(T,M,\eps,\beta) (|u|+\|r\|_{L^2(\Omega)}).
\end{equation}
\end{proposition}
\begin{proof}
First we define $L(T,h):=\frac{M\nabla T\cdot \nabla h}{\sqrt{\beta+|\nabla T|_M^2}}$. We easily see that \[\|L(T,h)\|_{L^2(\Omega)}\leq C(M)\|\nabla h\|_{L^2(\Omega)}\] holds true.
Thus Proposition \ref{prop:isom} gives us the existence of $v\in V$ satisfying \eqref{eq:kk1} and we have the estimate
\begin{multline*}
\|v\|_{H^1_0(\Omega \cap\Gamma_N)}\leq C(M,T,\eps)(\|g\|_{W^{2,6}(\Omega)}+\|L(T,g)\|_{L^2(\Omega)}+\|r\|_{L^2(\Omega)})\\
\leq C(M,T,\eps)(\|g\|_{W^{2,6}(\Omega)}+\|r\|_{L^2(\Omega)})\\
\leq C(M,T,\eps)(|u|+\|r\|_{L^2(\Omega)}).
\end{multline*}
This implies that $\delta T=v+g$ is the unique weak solution  of \eqref{lin_state_eq_eps}. Moving the term $L(T,v)$ to  the righthand side of \eqref{lin_state_eq_eps}, we conclude with standard elliptic regularity that $\delta T\in W$ and that \eqref{eq:kk2} holds.
\end{proof}

\begin{proposition}
Equation \eqref{ad_state_eq_eps1} has a unique weak solution  which satisfies $\phi \in H^2(\Omega)\cap V$ and
\[
\|\phi\|_{H^2(\Omega)}\leq C(M,T,\eps)\|h\|_{H^{1/2}(\Gamma_N)}.
\]
\end{proposition}
\begin{proof}
Proposition \ref{prop:isom} implies the existence of a weak solution which satisfies the estimate
\[
\|\phi\|_{H^1_0(\Omega\cup\Gamma_N)}\leq C(M,T,\eps)\|h\|_{H^{1/2}(\Gamma_N)}.
\]
Moving the $\div$-term to the righthand side of \eqref{ad_state_eq_eps1} and using
\[
\left\|\div\left(\frac{M\nabla T\phi}{\sqrt{\beta+|\nabla T|_M^2}}\right)\right\|_{L^2(\Omega)}\leq C(M)(\|T\|_{H^2(\Omega)}+1)\|\phi\|_{H^1_0(\Omega\cup\Gamma_N)}
\]
which follows from
\[
\div\left(\frac{M\nabla T\varphi}{\sqrt{\beta+\|\nabla T\|_M^2}}\right)=\div\left(\frac{M\nabla T}{\sqrt{\beta+\|\nabla T\|_M^2}}\right)\varphi+\frac{M\nabla T\cdot \nabla \phi}{\sqrt{\beta+\|\nabla T\|_M^2}}
\]
the claim follows from standard elliptic regularity.
\end{proof}

\section{Derivative of $J$}
In this section we characterize the gradient of $J$ using the linearized and adjoint state equations.
\begin{lemma}\label{lem:Lipschitz}
There exists a constant $C>0$ independent of $\beta>0$ such that
\[
\left|\sqrt{\beta+|\nabla T_1|_M^2}-\sqrt{\beta+|\nabla T_2|_M^2}\right|\leq |\nabla(T_1-T_2)|_M.
\]
holds.
\end{lemma}
\begin{proof}
There holds
\[
\sqrt{\beta+|\nabla T_1|_M^2}=|(\beta^{1/2},(M^{1/2}\nabla T_1)_1,\ldots,(M^{1/2}\nabla T_1)_d)|
\]
Using the reverse triangle inequality for $|\cdot|$ we get
\begin{multline*}
\left|\sqrt{\beta+|\nabla T_1|_M^2}-\sqrt{\beta+|\nabla T_2|_M^2}\right|\\
\leq |(0,(M^{1/2}\nabla (T_1-T_2))_1,\ldots,(M^{1/2}\nabla (T_1-T_2))_d)|\\
=|\nabla (T_1-T_2)|_M.
\end{multline*}
\end{proof}
\begin{lemma}\label{lem:f_frechet}
The function $f\colon H^2(\Omega)\to L^2(\Omega)$ defined by $f(T):=\sqrt{\beta+|\nabla T|_M^2}$ is Frechet differentiable with derivative
\[
f'(T)h=\frac{M\nabla T\cdot \nabla h}{f(T)}.
\]
\end{lemma}
\begin{proof}
By multiplication with the conjugate square root we get
\begin{multline*}
|f(T+h)-f(T)-f'(T)h|=\left|\frac{f(T+h)^2-f(T)^2}{f(T+h)+f(T)}-\frac{M\nabla T\cdot\nabla h}{f(T)}\right|\\
=\left|\frac{|\nabla T|_M^2+2M\nabla T\cdot\nabla h+|\nabla h|_M^2-|\nabla T|_M^2}{f(T+h)+f(T)}-\frac{M\nabla T\cdot\nabla h}{f(T)}\right|\\
=\left|\frac{(2M\nabla T\cdot \nabla h+|\nabla h|_M^2)f(T)-M\nabla T\cdot\nabla h\,(f(T+h)+f(T))}{f(T+h)f(T)+f(T)^2}\right|\\
=\left|\frac{|\nabla h|_M^2f(T)+M\nabla T\cdot \nabla h\,(f(T)-f(T+h))}{f(T+h)f(T)+f(T)^2}\right|\\
\leq \frac{|\nabla h|_M^2}{f(T)}+\frac{f(T)|\nabla h|^2_M}{f(T)^2}\leq \frac{2}{f(T)}|\nabla h|^2_M
\end{multline*}
utilizing Lemma \ref{lem:Lipschitz} and $\|\nabla T\|_M\leq \sqrt{\beta+\|\nabla T\|_M^2}=f(T)$.
Then using the embedding $H^2(\Omega)\hookrightarrow W^{1,4}(\Omega)$ and that $f(T)^{-1}\in L^{\infty}(\Omega)$ we get
\begin{multline*}
\frac{\|f(T+h)-f(T)-f'(T)h\|_{L^2(\Omega)}}{\|h\|_{H^2(\Omega)}}\leq C(M,T)\frac{\|h\|_{W^{1,4}(\Omega)}^2}{\|h\|_{H^2(\Omega)}}\\
\leq C(M,T)\|h\|_{H^2(\Omega)}.
\end{multline*}
\end{proof}
\begin{theorem}
The operator $S\colon \R^N\to W$, $u\mapsto T$ is Frechet differentiable and its derivative $S'(u)\delta u$ in direction $\delta u\in \R^N$ is given by
the solution $\delta T\in W$ of \eqref{lin_state_eq_eps} with $u_i=\delta u_i$ for $i=1,\ldots,N$ and $r=0$.
\end{theorem}
\begin{proof}
We introduce the mapping $E\colon W\times \R^N\to L^2(\Omega)\times \mathbb R^N$ defined by
\[
E(T,u)=\left(
         \begin{array}{c}
           -\eps\div(M\nabla T)+\sqrt{\beta+|\nabla T|_M^2}-1 \\
           T|_{\Gamma_1}-u_1\\
           \vdots \\
           T|_{\Gamma_N}-u_N
         \end{array}
       \right)
\]
Using Lemma \ref{lem:f_frechet} it can be argued that $E$ is Frechet differentiable. {\color{black}Moreover due to Proposition \ref{prop:tangent} the operator $D_TE(T,u)\colon W\to L^2(\Omega)\times \mathbb R^N$ given by
\[
D_TE(T,u)\delta T=\left(
         \begin{array}{c}
           -\eps\div(M\nabla \delta T)+\frac{M\nabla T\cdot \nabla \delta T}{\sqrt{\beta+|\nabla T|_M^2}}\\
           \delta T|_{\Gamma_1}\\
           \vdots \\
           \delta T|_{\Gamma_N}
         \end{array}
       \right)
\]
is an isomorphism.} Let $(T_0,u_0)\in W\times \R^N$ such that $E(T_0,u_0)=0$. Then there exists a neighbourhood $V\subseteq W$ of $T_0$ and $U\subseteq \R^N$ of $u_0$ and a Frechet differentiable implicit function $S\colon U\to V$, $u\mapsto T$ with  derivative given by $\delta T=D_TE(T,U)^{-1}(0,\delta u)$. Since $u_0$ is arbitrary, the result follows.
\end{proof}
\begin{theorem}\label{thm:derivative}
There holds
\[
\nabla J(u)=S'(u)^\ast(S(u)-z)=\left(\int_{\Gamma_i}-\eps M\nabla \phi\cdot n~\mathrm ds\right)_{i=1}^N.
\]
where $\phi$ solves \eqref{ad_state_eq_eps1} for $h=S(u)-z$.
\end{theorem}
\begin{proof}
For each $\delta u \in \mathbb{R}^N$ we have
\[
DJ(u)\delta u=\int_{\Gamma_N}(S(u)-z)S'(u)\delta u~\mathrm ds.
\]
There holds
\begin{multline*}
\int_{\Gamma_N}(S(u)-z)\,S'(u)\delta u ~\mathrm ds=\int_\Omega\eps M\nabla \delta T\cdot\nabla \phi+\frac{M\nabla T\cdot \nabla \delta T}{\sqrt{\beta+|\nabla T|_M^2}}\phi~\mathrm dx\\
-\sum_{i=1}^N\int_{\Gamma_i}\eps M\nabla \phi\cdot n\,\delta T~\mathrm ds\\
=\int_\Omega\left(-\eps\div(M\nabla \delta T)+\frac{M\nabla T\cdot\nabla \delta T}{\sqrt{\beta+|\nabla T|_M^2}}\right)\phi~\mathrm dx+\sum_{i=1}^N\int_{\Gamma_i}\eps M\nabla \delta T\cdot n\,\phi~\mathrm ds\\
+\int_{\Gamma_N}\eps M\nabla \delta T\cdot n\,\phi~\mathrm ds-\sum_{i=1}^N\int_{\Gamma_i}\eps M\nabla \phi\cdot n\,\delta T~\mathrm ds\\
=-\sum_{i=1}^N\int_{\Gamma_i}\eps M\nabla \phi\cdot n\,\delta u_i~\mathrm ds=(S'(u)^\ast(S(u)-z))\cdot \delta u
\end{multline*}
where $\delta T=S'(u)\delta u\in W$ solves \eqref{lin_state_eq_eps} with $r=0$.

\end{proof}

\section{A projected Levenberg Marquardt method}
We solve the inverse problem \eqref{inverse_prob_eps} based on a Levenberg Marquardt strategy. Let $P_{ad}\colon \R^d\to U_{ad}$ be the orthogonal projection on $U_{ad}$. In particular we iterate
\[
u_{k+1}=P_{\text{ad}}(u_{k}+\lambda d)
\]
where $0<\lambda\leq 1$ is the stepsize and $d$ solves the problem
\[
\min_{d\in \R^d}\frac 12\int_{\Gamma_N}(S(u_k)-z+S'(u_k)d)^2~\mathrm ds+\frac {\alpha_k} 2|d|^2=j(d).
\]
The gradient of $j$ is given by
\[
Dj(d)\delta d=\int_{\Gamma_N}(S(u_k)-z+S'(u_k)d)S'(u_k)\delta d~\mathrm ds+\alpha_kd\cdot \delta d.
\]
Thus we have to solve the equation
\[
(S'(u_k)^\ast S'(u_k)+\alpha_kI)d=-S'(u_k)^\ast(S(u_k)-z)
\]
Let $\mathcal H(u)$ be the matrix representation of $S'(u)^\ast S'(u)$.
\begin{proposition}
The matrix $\mathcal H(u)$ is positive definitive and there holds
\[
S'(u_k)^\ast S'(u_k)\delta u=\left(-\eps\int_{\Gamma_i}M\nabla w\cdot n~\mathrm ds\right)_{i=1}^N
\]
with $w = S'(u)^\ast S'(u)\delta u$ and $\delta u\in \mathbb \R^n$. 
\end{proposition}
\begin{proof}
The formula follows from the exact same calculation as in the proof of Theorem \ref{thm:derivative}, where we replace $T-z$ by $S'(u)\delta u$ and $\varphi$ by $w$.
Moreover we have
\[
\mathcal H(u)\delta u\cdot \delta u=\int_{\Gamma_N}(S'(u)\delta u)^2~\mathrm ds\geq 0.
\]
The corresponding equality implies $S'(u)\delta u=0$ on $\Gamma_N$. This fact, together with the unique continuation principle \cite{arrv09,Salo14} and uniqueness of solutions for the linearized state equation \eqref{lin_state_eq_eps} imply that  $\delta u=0$.
\end{proof}

\section{Numerical example}
In this section we present two numerical examples. In the first one we reconstruct the the activation instants using the proposed Levenberg Marquardt method. In the second example we jointly reconstruct the positions of the activation regions and the activation instants using a combined shape gradient and projected gradient method.
\subsection{Finding the activation instants}
In this example, the computational domain $U$ is given by the unit-square $(0,1)\times (0,1)$.
We consider three activation sites $\omega_i=B_{0.1}(x_i)$ whose midpoints are given
by $x_1=(0.5,0.8)^\top$, $x_2=(0.2,0.2)^\top$ and $x_3=(0.8,0.4)^\top$. Thus we have
$\Omega=U\setminus \bigcup_{i=1}^3\omega_i$. The admissible set is given by $U_{ad}=\{u\in \mathbb R^3|\quad u_i\geq 0\quad i=1,2,3\}$. The observed data is given on the boundary $\Gamma_N$ of $U$.
The domain $U$ is discretized by 66049 vertices and 131072 triangles, which yields a
discretization size of $\approx 4\cdot10^{-3}$. The state and adjoint variable are approximated by $P1$ finite elements on the mentioned grid using the Fenics toolbox. Moreover we set $\eps=0.1$, $\beta=0$ and
\[
M=\left(
    \begin{array}{cc}
      \sin(\pi x)+1.1 & 0 \\
      0 & \sin(\pi y)+1.1\\
    \end{array}
  \right).
\]
The case $\beta=0$ is not considered in the theoretic part of this work. However this case is very important from a practical point of view. Moreover the proposed method also works in this case. The exact activation instants are given by $u^\dagger=(0,0.1,0.2)^\top$. Then observed data $z$ is generated by solving the state equation for $T$ for $u^\dagger$, restricting $T$ to $\Gamma_N$ and adding noise $\eta$. The used perturbance has the form
\[
\eta = \delta \|S(u^\dagger)\|_{L^2(\Gamma_N)}\frac{\hat \eta|_{\Gamma_N}}{\|\hat \eta\|_{L^2(\Gamma_N)}},
\]
where $\delta \geq 1$ and  $\hat \eta$ is a FEM-function with random coefficients on $\bar\Omega$. The random coefficients are chosen from a standard normal distribution. Thus $\delta$ is the relative noise level. In this example we choose $\delta=0.1$ and $\delta=10^{-9}$.\\
In every step of the Gauss-Newton iteration the matrix $\mathcal H(u)$ is calculated by solving the linearized state equation and the adjoint equation for all unit vectors resulting in a $3\times3$ matrix. Thus 6 linear PDEs must be solved. Moreover the gradient of $J$ has to be calculated by solving the nonlinear state equation and the adjoint state equation. So in complete 8 PDEs has to be solved per iteration. The nonlinear state equation is solved by the Newton method. Since $\beta =0$ the method can be interpreted as a semi smooth Newton method. The iteration is stopped by the discrepancy criterium
\[
\|S(u_{K_\delta})-z\|_{L^2(\Gamma_N)}\leq \tau \delta \leq \|S(u_{k})-z\|_{L^2(\Gamma_N)}\quad\forall~ 0\leq k<K_\delta
\]
with $\tau > 1$, see \cite{ClasonHuu}. In our experiments we choose $\tau=1.1$. The parameter $\alpha_k$ is set to $0.1^k$.\\

In the case $\delta = 0.1$ the discrepancy criterium is satisfied after 2 iterations with a final iterate $u_{K_\delta}=(0.0004,0.0981,0.1860)^\top$, the state error $\|S(u_{K_\delta})-z\|_{L^2(\Gamma_N)}=0.0494$ and $|u_{K_\delta}-u^\dagger|=0.0141$. In the case $\delta = 10^{-9}$ the discrepancy criterium is satisfied after 5 iterations with a final iterate $u_{K_\delta}=(2.2\cdot 10^{-11},0.1,0.2)^\top$, the state error $\|S(u_{K_\delta})-z\|_{L^2(\Gamma_N)}=4.8\cdot 10^{-10}$ and $|u_{K_\delta}-u^\dagger|=1.1\cdot 10^{-10}$. So we can observe that the activation instants are reconstructed very well relative to the noise level.
\subsection{Finding the activation instants and activation regions}
In this section we consider a similar scenario as before. But in addition to the activation instants  we also reconstruct the position of the activation regions $\omega_i$ by determining  the midpoints of $\omega_i$. For this purpose we use the shape optimization approach introduced in \cite{kunisch2019inverse} for the squared version of the Eikonal equation. Here we only modify the formulas developed in that work to fit our state equation. The shape derivative of $J$ with respect to a smooth perturbation field $h$ with compact support on $U=\Omega\cup \bigcup_{i=1}^n\bar \omega_i$ is given by
\begin{equation}\label{shape_diff}
DJ(\Omega,\Gamma)h= \int_{\Omega}S_1\colon Dh+S_0\cdot h~\mathrm dx
\end{equation}
for any $h\in \mathcal C_c^\infty(U,\R^d)$, where $S_i$, $i=0,1$ have the form
\begin{align*}
S_1&=\mathrm {Id}_{\R^d}\left(\eps M\nabla T\cdot \nabla \phi+\left(\|\nabla T\|_M-1\right)\phi\right)\\
&-\eps(\nabla T\otimes M\nabla \phi+\nabla \phi\otimes M\nabla T)\nonumber
-\frac{\nabla T\otimes M\nabla T}{\|\nabla T\|_M}\phi,\\
S_0&=\eps M_{\nabla T}^\ast\nabla \phi+\frac{M_{\nabla T}^\ast\nabla T}{2\|\nabla T\|_M}\phi,
\end{align*}
with the outer product $v\otimes w=vw^\top$ for $v,w\in \R^d$, the inner product $G\colon N=\text{trace}(GN^\top)$ for $G,N\in \R^{d\times d}$, and
$$M_vh=\left(\sum_{k=1}^dDM_kv_k\right)h,$$
where $M_k$ stands for the k-th column of $M$.
Based on the shape derivative we calculate a perturbation field $h$ by solving a linear elasticity equation of the
form
\begin{equation}
\label{H1_representation}
\int_U \gamma Dh \colon Dv+h\cdot v~\mathrm dx=-\int_\Omega S_1 \colon Dv+S_0\cdot v~\mathrm dx, \quad\forall v\in  H^1_{0}(U,\R^d)
\end{equation}
for $\gamma>0$ and thus $h$ is a decent direction for $J$. Since we are only interested in the shift of
the midpoints $x_i$ of $\omega_i$, we average $h$ over $\omega_i$, $i=i,\ldots,N$,
in order to get a shift of the midpoints. The proposed method is of gradient type and thus we also update the activation instants based on the gradient calculated in Theorem \ref{thm:derivative}. In particular we use a projected gradient method.

In the specific example we choose the exact activation sites as $\omega_i^\dagger=B_{0.05}(x_i)$ with $x_1^\dagger=(0.5,0.8)^\top$, $x_2^\dagger=(0.2,0.3)^\top$ and $x_3^\dagger=(0.7,0.4)^\top$. We denote $X^\dagger=[x_1^\dagger,x_2^\dagger,x_3^\dagger]$.  The exact activation instants are given by $u^\dagger=(0,0.1,0.2)$.

We start the iteration at the initial points $x_1^0=(0.2,0.8)^\top$, $x_2^0=(0.2,0.2)^\top$ and $x_3^0=(0.8,0.2)^\top$ and initial times $u^0=(0,0,0)$. Relative noise levels are chosen to be $\delta=0.1, 0.01, 0.001$ and the iteration is stopped using the discrepancy criterium.

%
%

In Table \ref{tab:error} and Figure \ref{fig:error} we summarize our finding for the three noise levels $\delta$. In particular we document the number of iterations $K_\delta$ at which the discrepancy criterion is reached, the state error $\|S(X_{K_\delta},u_{K_\delta})-z\|_{L^2(\Gamma_N)}$, the distance between reconstructed and exact positions denoted  by $d_{K_\delta}$ and the reconstruction error $|u_{K_\delta}-u^\dagger|$.
The reconstructed position of the three midpoints as well as the activation instants $u$ are given for the respective noise levels by
\begin{align*}
X_{K_{.1}}&= [(0.396,0.809),(0.22,0.275),(0.72, 0.352)],\\
X_{K_{.01}}&= [(0.496,0.821),(0.195,0.296),(0.711, 0.407)],\\
X_{K_{.001}}&= [(0.499,0.803),(0.201,0.301),(0.7, 0.408)]
\end{align*}
as well as
\begin{align*}
u_{.1}&=(0.038,0.113,0.171),\\
u_{.01}&=(0.011,0.103,0.193),\\
u_{.001}&=(0.003,0.099,0.196).
\end{align*}

We conclude that the positions can be reconstructed with good quality relative to the noise level. Further tests in the noise free case showed that there is limit until which the state error can be reduced. This is caused by discretization effects.
\begin{table}[ht]
\centering
\begin{tabular}{|c|c|c|c|c|c|c|}
  \hline
   $\delta$ & $K_\delta$ & $\|S(X_{K_\delta},u_{K_\delta})-z\|_{L^2(\Gamma_N)}$ & $d_{K_\delta}^1$ & $d_{K_\delta}^2$ & $d_{K_\delta}^3$ & $|u_{K_\delta}-u^\dagger|$ \\\hline
  $0.1$ & $2$ & $0.108$ & $0.104$ & $0.032$ & $0.052$ & $0.049$ \\\hline
  $0.01$ & $9$ & $0.0104$ & $0.021$ & $0.006$ & $0.013$ & $0.013$  \\\hline
  $0.001$ & $52$ & $0.001$ & $0.003$ & $0.001$ & $0.008$ & $0.005$ \\
  \hline
\end{tabular}
\caption{Iterations and reconstruction errors for different $\delta$}\label{tab:error}
\end{table}
\begin{figure}[ht]
  \centering
%
%
\definecolor{mycolor1}{rgb}{0.00000,0.44700,0.74100}%
\begin{tikzpicture}

\begin{axis}[%
width=4.8cm,
height=4.8cm,
at={(0.758333in,0.48125in)},
scale only axis,
xmin=0,
xmax=60,
ymode=log,
ymin=0.001,
ymax=1,
yminorticks=true,
legend style={legend cell align=left,align=left,draw=white!15!black}
]
\addplot [color=mycolor1,solid]
  table[row sep=crcr]{%
0	0.36415441042518\\
1	0.170509201001369\\
2	0.089608650182219\\
3	0.0500888817841981\\
4	0.0290104540219312\\
5	0.0186254795820508\\
6	0.0134019779354252\\
7	0.010952812833018\\
8	0.009288159881528\\
9	0.00840528395030488\\
10	0.00749475083783855\\
11	0.00679805946170466\\
12	0.00606519523230689\\
13	0.0055241538112366\\
14	0.00498194587933552\\
15	0.00463227972486803\\
16	0.00430046996976196\\
17	0.0038534504746961\\
18	0.00358314193714274\\
19	0.00336722768778781\\
20	0.00310473498315884\\
21	0.00294229397105261\\
22	0.00268409448388414\\
23	0.00259145891954592\\
24	0.00241840271110609\\
25	0.00227235898475382\\
26	0.00215869395067012\\
27	0.00208330128912832\\
28	0.00194759722599721\\
29	0.00185445886006101\\
30	0.00178556460126568\\
31	0.00172488082764639\\
32	0.00168574800528947\\
33	0.00166947218535398\\
34	0.00157805943479836\\
35	0.00151684690682275\\
36	0.00149966449102907\\
37	0.0014274017907147\\
38	0.00140251616964861\\
39	0.00136824924222294\\
40	0.00126489664184231\\
41	0.00126053782791266\\
42	0.00124205810743655\\
43	0.00120800468903231\\
44	0.00120287100066465\\
45	0.00118630951769809\\
46	0.00117539827523405\\
47	0.00123974406806226\\
48	0.00118459888627402\\
49	0.00115881520826802\\
50	0.0011464633984126\\
51	0.00110110560490433\\
52	0.00108777303384046\\
};
\addlegendentry{\tiny $\|S(X_{k},u_{k})-z\|_{L^2(\Gamma_N)}$};

\end{axis}
\end{tikzpicture}%
%
%
\begin{tikzpicture}

\begin{axis}[%
width=4.8cm,
height=4.8cm,
at={(0.758333in,0.48125in)},
scale only axis,
xmin=0,
xmax=1,
ymin=0,
ymax=1
]
\addplot [color=blue,only marks,mark=o,mark options={solid}]
  table[row sep=crcr]{%
0.2	0.8\\
};
\addlegendentry{\tiny $X_0$};

\addplot [color=blue,only marks,mark=asterisk,mark options={solid}]
  table[row sep=crcr]{%
0.5	0.8\\
};
\addlegendentry{\tiny $X^\dagger$};

\addplot [color=blue,solid,forget plot]
  table[row sep=crcr]{%
0.2	0.8\\
0.304457968771179	0.793918239742577\\
0.395034746687251	0.805756528330713\\
0.445298071987369	0.817438949873219\\
0.470409844546148	0.823231244010901\\
0.48300500354728	0.82508136054528\\
0.489767050281563	0.82481924843119\\
0.493159698398063	0.823767890353605\\
0.4948574521235	0.822153035541733\\
0.495916392030569	0.820297425154796\\
0.496819538295511	0.81865955845405\\
0.497434164750856	0.817114392099874\\
0.49741977639144	0.81579417278774\\
0.497679058891005	0.81442841811495\\
0.498170019436624	0.813177465386813\\
0.498224574788414	0.812367616131618\\
0.498335799135867	0.811498054414175\\
0.498385541037913	0.810475105231084\\
0.498299871634374	0.809762208663314\\
0.498407064674239	0.809091603963219\\
0.498581794101938	0.80866631238529\\
0.498710627170603	0.808188492482103\\
0.498808826730461	0.80765811671036\\
0.498835705535152	0.807164665854676\\
0.498734567257067	0.806639931771363\\
0.498710763399174	0.806306173110253\\
0.498743097508981	0.805941261575714\\
0.498822096579202	0.805619956880299\\
0.498940268193466	0.805474228873687\\
0.499139838143714	0.805468649518772\\
0.499121212400464	0.805316775794605\\
0.499191881407663	0.805172550788314\\
0.499176708914904	0.805088774523128\\
0.499221796347418	0.805053969820246\\
0.499168392650042	0.804987494100951\\
0.499181251677257	0.804757125783601\\
0.499236124329713	0.804512864363276\\
0.499210425400748	0.804332024031315\\
0.499329245771969	0.804084138309656\\
0.499190217304011	0.803882823310183\\
0.499295436161727	0.803684115495038\\
0.499249577014276	0.803639564040186\\
0.499193075493572	0.803564548757898\\
0.499225986830156	0.803401720013719\\
0.499139734895822	0.803368510246516\\
0.499167714731614	0.80321924414974\\
0.49915786655186	0.803118662221914\\
0.499201188767265	0.803050087069374\\
0.499097014767606	0.803143055579572\\
0.499139001815281	0.803026952158125\\
0.499137865713366	0.802995399194329\\
0.499116477004867	0.802949249171563\\
0.49912937044895	0.802899988755844\\
};
\addplot [color=blue,solid,forget plot]
  table[row sep=crcr]{%
0.2	0.2\\
0.217971841608386	0.249770558534253\\
0.218654855286902	0.275566032096932\\
0.211211544649339	0.283108873134198\\
0.204086433189642	0.286880496421208\\
0.199623190447715	0.289992285149245\\
0.197055394665656	0.292120275439662\\
0.195773574627805	0.294029471220632\\
0.19517240671471	0.295289876713476\\
0.195118807261723	0.29636517074007\\
0.19519139364129	0.296960080192948\\
0.195526412113408	0.297601071276241\\
0.195800776956265	0.298097188340934\\
0.196333291825935	0.298512661094826\\
0.19691911949485	0.29886327272142\\
0.197283775908489	0.299198204507715\\
0.197604113303148	0.299323912206098\\
0.197912283214199	0.299491769486089\\
0.198237283952285	0.299732202438193\\
0.198532422982329	0.29984126399983\\
0.198653154059324	0.29998309482611\\
0.198826665448983	0.300102774038623\\
0.198934248001947	0.299981687456039\\
0.199133286923931	0.300184049237167\\
0.199297390557614	0.300181772209721\\
0.19949762907789	0.300254509503975\\
0.199762505764583	0.300318895953554\\
0.199892802115663	0.300264064329724\\
0.200053670692847	0.300331714446205\\
0.200198638236499	0.300450882189737\\
0.200280319621944	0.300523422867205\\
0.200348044019895	0.300617214285829\\
0.200413315012134	0.300712470485226\\
0.200485643644741	0.3008266916748\\
0.200492474669761	0.300900547598317\\
0.200543670356687	0.300941737732742\\
0.200594784082784	0.300987878556464\\
0.200621911089784	0.301010211430396\\
0.200687825688702	0.301028943596625\\
0.200702015846634	0.301096644384201\\
0.200744049861815	0.301109163133709\\
0.200762456964001	0.301122669895761\\
0.200778509817551	0.301130745890363\\
0.200778136926928	0.301049215901635\\
0.200801910773048	0.301084856434765\\
0.200837399726451	0.301060834794226\\
0.200853446424223	0.301038598149802\\
0.2008751532232	0.301009150093108\\
0.200866638105012	0.301029128577576\\
0.200885555287411	0.301003876011471\\
0.200880882866327	0.300991428442594\\
0.200869483664149	0.300973781407118\\
0.200864499771551	0.300948138424679\\
};
\addplot [color=blue,solid,forget plot]
  table[row sep=crcr]{%
0.8	0.2\\
0.739032001285941	0.290797423642729\\
0.719754827345687	0.352550077630236\\
0.718139837099816	0.378979579623271\\
0.717042092798461	0.391639989653093\\
0.715443267136817	0.39856360033036\\
0.714261804763767	0.402255630775061\\
0.713002672315375	0.404596893979977\\
0.711797008272157	0.406007888098887\\
0.710862561784015	0.406830712904403\\
0.710123063922589	0.407356338622434\\
0.709302473251861	0.40783499403802\\
0.708384234689597	0.408219276088342\\
0.707760739107248	0.408413582232397\\
0.707267761996959	0.408564539736375\\
0.706635970838901	0.408777396898057\\
0.706068410143812	0.408908033211486\\
0.705501177699951	0.408992792470201\\
0.704996770128583	0.409056510622489\\
0.704570520236615	0.409052685142117\\
0.704175920704465	0.40911861501286\\
0.703764467475734	0.409184287572123\\
0.70346665959078	0.409104231802968\\
0.703130038165351	0.409124551270177\\
0.702759011883453	0.409075011887793\\
0.702424331265073	0.409072561208287\\
0.702094218427468	0.409070138347419\\
0.701815445937816	0.4090162517811\\
0.701536083158009	0.409039806493325\\
0.701342062784772	0.40905801063046\\
0.701122521695498	0.409010759635623\\
0.700879454404218	0.409022251290532\\
0.700653313744449	0.409007865263542\\
0.700472501111975	0.4090005296579\\
0.700280786937053	0.408974826965675\\
0.700081198087555	0.408925925580347\\
0.699926932307058	0.408871740571522\\
0.69985911523054	0.408752920126408\\
0.699748276740156	0.408689116110473\\
0.699534836267588	0.408667579975995\\
0.699344626553708	0.408668425449689\\
0.699210241146786	0.408613044391505\\
0.699062057909929	0.408559101770028\\
0.698932946441648	0.408478403377791\\
0.698803425093481	0.408430594625881\\
0.698673630824689	0.408367257929087\\
0.698552523349376	0.408299078156718\\
0.69843306089719	0.408240707060451\\
0.698317155761314	0.408193031360526\\
0.698180861677109	0.408146406881574\\
0.698030532800785	0.408123443047599\\
0.697876227535258	0.408097391722684\\
0.697776511873067	0.408042327466369\\
};
\addplot [color=blue,only marks,mark=o,mark options={solid},forget plot]
  table[row sep=crcr]{%
0.2	0.2\\
};

\addplot [color=blue,only marks,mark=o,mark options={solid},forget plot]
  table[row sep=crcr]{%
0.8	0.2\\
};

\addplot [color=blue,only marks,mark=asterisk,mark options={solid},forget plot]
  table[row sep=crcr]{%
0.2	0.3\\
};

\addplot [color=blue,only marks,mark=asterisk,mark options={solid},forget plot]
  table[row sep=crcr]{%
0.7	0.4\\
};

\node[] at (axis cs: .45,.9) {\tiny $1$};
\node[] at (axis cs: .3,.3) {\tiny $2$};
\node[] at (axis cs: .8,.35) {\tiny $3$};
\end{axis}
\end{tikzpicture}%
\\
%
%
\definecolor{mycolor1}{rgb}{0.00000,0.44700,0.74100}%
\definecolor{mycolor2}{rgb}{0.85000,0.32500,0.09800}%
\definecolor{mycolor3}{rgb}{0.92900,0.69400,0.12500}%
\begin{tikzpicture}

\begin{axis}[%
width=4.8cm,
height=4.8cm,
at={(0.758333in,0.48125in)},
scale only axis,
xmin=0,
xmax=60,
ymode=log,
ymin=0.0001,
ymax=1,
yminorticks=true,
legend style={legend cell align=left,align=left,draw=white!15!black}
]
\addplot [color=mycolor1,solid]
  table[row sep=crcr]{%
0	0.670820393249937\\
1	0.195636586008145\\
2	0.10512298521937\\
3	0.0574144398298867\\
4	0.0376200478213036\\
5	0.0302969396347272\\
6	0.0268460118570554\\
7	0.0247326168827012\\
8	0.0227420927439708\\
9	0.0207041377980957\\
10	0.0189286676328214\\
11	0.0173056617173106\\
12	0.016003544854786\\
13	0.0146138980744783\\
14	0.013303925093106\\
15	0.0124944012846147\\
16	0.0116178664059938\\
17	0.0105987880129077\\
18	0.0099091449906484\\
19	0.00923009780967043\\
20	0.00878158746063179\\
21	0.00828938426076616\\
22	0.00775020292041204\\
23	0.00725865128037163\\
24	0.00675943887872626\\
25	0.00643661015670432\\
26	0.00607275826811678\\
27	0.00574207034136215\\
28	0.00557585986761596\\
29	0.00553588348668688\\
30	0.00538891199549538\\
31	0.00523529725201574\\
32	0.00515494271142117\\
33	0.00511353223015627\\
34	0.00505634929484894\\
35	0.00482706893849677\\
36	0.00457705700215551\\
37	0.00440339190348458\\
38	0.004138852131548\\
39	0.00396636670932315\\
40	0.00375088218729662\\
41	0.00371612177681467\\
42	0.00365474144181299\\
43	0.00348866671363572\\
44	0.00347662441606102\\
45	0.00332509122635855\\
46	0.00323036264200623\\
47	0.00315295583798847\\
48	0.00327019582077304\\
49	0.00314702355276614\\
50	0.00311700045902727\\
51	0.00307874707614859\\
52	0.00302785907848815\\
};
\addlegendentry{\tiny $d_k^1$};

\addplot [color=mycolor2,solid]
  table[row sep=crcr]{%
0	0.5\\
1	0.0533477635965915\\
2	0.0307412168474741\\
3	0.0202733544392812\\
4	0.0137411902820445\\
5	0.0100148061375506\\
6	0.00841194149538893\\
7	0.00731504513533571\\
8	0.00674469556780935\\
9	0.00608588746996243\\
10	0.00568891971698767\\
11	0.00507620405428437\\
12	0.00461022411395732\\
13	0.00395688335089664\\
14	0.00328389609958705\\
15	0.0028320927468487\\
16	0.00248945129882362\\
17	0.00214868788635898\\
18	0.00178294234314564\\
19	0.00147613665375022\\
20	0.0013469520306309\\
21	0.00117782701260598\\
22	0.00106590931632105\\
23	0.000886039320758666\\
24	0.000725741803093014\\
25	0.00056316217113294\\
26	0.000397615569425516\\
27	0.000284993608066944\\
28	0.0003360282980519\\
29	0.000492698587395416\\
30	0.000593759705908023\\
31	0.0007085817626895\\
32	0.000823676812574598\\
33	0.000958785103587064\\
34	0.0010264098972591\\
35	0.0010874039792136\\
36	0.00115311410860108\\
37	0.0011862969854547\\
38	0.00123767083793451\\
39	0.00130209644586099\\
40	0.00133560961888047\\
41	0.00135710291312162\\
42	0.00137283058117017\\
43	0.00130627374056596\\
44	0.00134906425791455\\
45	0.00135152090716375\\
46	0.00134426817108466\\
47	0.00133576834612028\\
48	0.00134542455538816\\
49	0.00133864678443147\\
50	0.00132622961095444\\
51	0.0013054700575157\\
52	0.00128309248589674\\
};
\addlegendentry{\tiny $d_k^2$};

\addplot [color=mycolor3,solid]
  table[row sep=crcr]{%
0	0.5\\
1	0.115968529383843\\
2	0.0513979409738811\\
3	0.0277652978161995\\
4	0.0189821679465686\\
5	0.0155099240446247\\
6	0.0144390770242846\\
7	0.0137913350189225\\
8	0.0132387357244592\\
9	0.0128317530873189\\
10	0.0125136781606665\\
11	0.0121623657311151\\
12	0.0117410344836661\\
13	0.0114462848764161\\
14	0.0112326179201639\\
15	0.0110035814751723\\
16	0.0107786203834471\\
17	0.0105419766883894\\
18	0.0103435050332635\\
19	0.0101410435121644\\
20	0.0100293296626813\\
21	0.00992584271401291\\
22	0.00974190769000222\\
23	0.00964647991749535\\
24	0.0094851456149406\\
25	0.00939088648429546\\
26	0.00930876793476313\\
27	0.00919720827932651\\
28	0.0091693867245898\\
29	0.00915689298287332\\
30	0.00908040991188413\\
31	0.00906501287360388\\
32	0.00903152564381099\\
33	0.00901292357803828\\
34	0.00897921824925879\\
35	0.00892629489739799\\
36	0.00887204145933362\\
37	0.00875405387563669\\
38	0.00869276155090155\\
39	0.00868005299167162\\
40	0.00869316479372833\\
41	0.0086491764195367\\
42	0.00861034020663645\\
43	0.00854528683737034\\
44	0.00851508762449635\\
45	0.00847173302465867\\
46	0.0084243627002528\\
47	0.00838835806389359\\
48	0.00836407362510871\\
49	0.00834704794031762\\
50	0.00835877550819265\\
51	0.00837127004656928\\
52	0.00834403562593848\\
};
\addlegendentry{\tiny $d_k^3$};

\end{axis}
\end{tikzpicture}%
%
%
\definecolor{mycolor1}{rgb}{0.00000,0.44700,0.74100}%
\begin{tikzpicture}

\begin{axis}[%
width=4.8cm,
height=4.8cm,
at={(0.758333in,0.48125in)},
scale only axis,
xmin=0,
xmax=60,
ymode=log,
ymin=0.001,
ymax=1,
yminorticks=true,
legend style={legend cell align=left,align=left,draw=white!15!black}
]
\addplot [color=mycolor1,solid]
  table[row sep=crcr]{%
0	0.223606797749979\\
1	0.077923267277587\\
2	0.0505578160877107\\
3	0.037857769833113\\
4	0.0284215667342058\\
5	0.0225026782518495\\
6	0.0187209275367033\\
7	0.0161706143789638\\
8	0.014343325296741\\
9	0.0131990717068322\\
10	0.0121048090924291\\
11	0.0113299280242798\\
12	0.0105380103383962\\
13	0.0100359047662585\\
14	0.0096639908646021\\
15	0.00910934316870346\\
16	0.0086614924442655\\
17	0.00842209365574438\\
18	0.00810878102030136\\
19	0.00786473340551233\\
20	0.00751485713329083\\
21	0.00725503632278931\\
22	0.0070889562065248\\
23	0.00693349974243082\\
24	0.00684394630693235\\
25	0.00669265467487973\\
26	0.00659657749296741\\
27	0.00652413087652875\\
28	0.00637264437182932\\
29	0.00619916166283071\\
30	0.00609597934548207\\
31	0.00598205377745128\\
32	0.00586686099017582\\
33	0.00573937407488384\\
34	0.0056226860707128\\
35	0.00558847628979708\\
36	0.00556088350756066\\
37	0.00554484024918701\\
38	0.00554079179029266\\
39	0.00549181522127042\\
40	0.00546818835629283\\
41	0.00543214058365236\\
42	0.00540498747337386\\
43	0.00543268914199043\\
44	0.00539222480613948\\
45	0.00540840099591279\\
46	0.00541073986734382\\
47	0.00541163937165576\\
48	0.00534701222599367\\
49	0.00535066836623385\\
50	0.00532262579410956\\
51	0.00530113708004671\\
52	0.00529757270212716\\
};
\addlegendentry{\tiny $|u_k-u^\dagger|$};

\end{axis}
\end{tikzpicture}%
  \caption{
  top left: State error during the iteration; top right: Paths of the midpoints in $U$ during the iteration; bottom left: Distances between $X_k$ and $X^\dagger$ during the iteration; bottom right: Error of the activation instants during the iteration
  }\label{fig:error}
\end{figure}
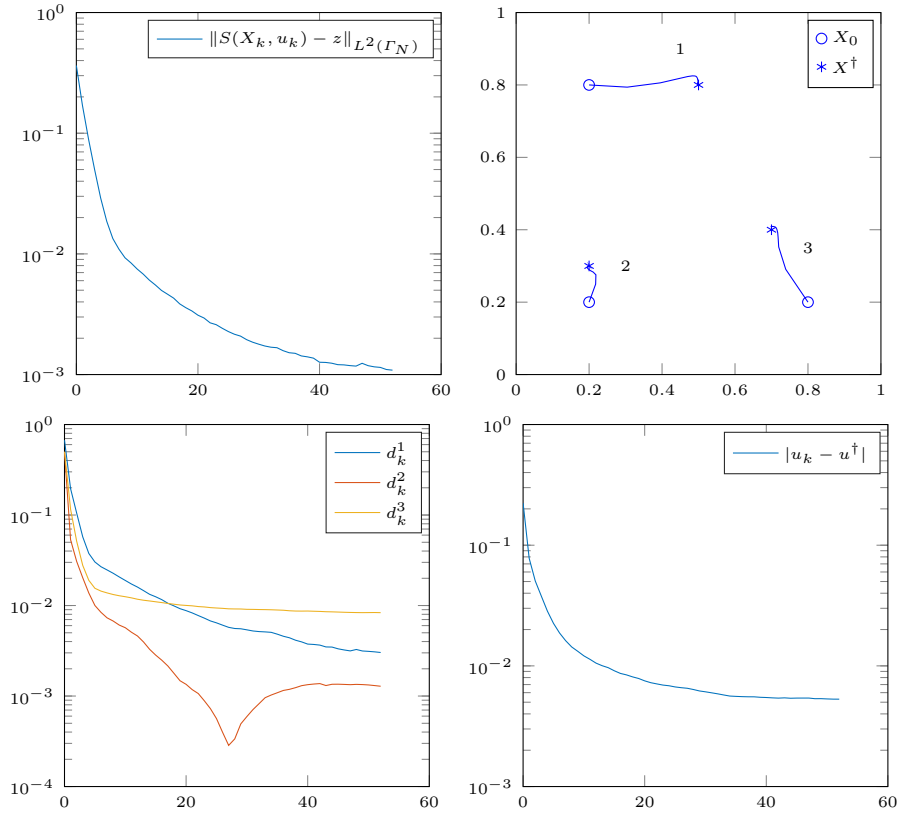

\bibliographystyle{spmpsci}
\bibliography{lit}

\end{document}